\documentclass[11pt]{amsart}
\usepackage[utf8]{inputenc}
\usepackage{graphicx,amssymb,amsmath,amsthm,bm}
\usepackage{caption}
\usepackage{subcaption}
\usepackage{indentfirst}
\usepackage{cancel}
\usepackage{lipsum}
\usepackage{epstopdf,epsfig}
\usepackage{enumerate,color} 
\usepackage{cases}
\usepackage{comment}
\usepackage[normalem]{ulem}
\usepackage{array,ragged2e}
\usepackage{ulem}
\newcolumntype{C}{>{\Centering\arraybackslash}m{0.14\linewidth}}

\usepackage[top=3cm, bottom=1in, left=1in, right=1in]{geometry}

\usepackage{accents}

\usepackage[overload]{empheq}
\usepackage{soul}
\usepackage[multiple]{footmisc}
\newcommand{\mbf}{\mathbf}

\newcommand{\veps}{\varepsilon}

\newcommand{\mbfu}{\mathbf{u}}
\newcommand{\mbfx}{\mathbf{x}}
\newcommand{\mbfy}{\mathbf{y}}
\newcommand{\mbfc}{\mathbf{c}}
\newcommand{\R}{\mathbb{R}}

\DeclareMathOperator{\supp}{supp}

\numberwithin{equation}{section}
\everymath{\displaystyle}

\theoremstyle{plain}
\newtheorem{theorem}{Theorem}[section]
\newtheorem{lemma}[theorem]{Lemma}

\theoremstyle{definition}

\newtheorem*{defi*}{Definition} 
\theoremstyle{remark}
\newtheorem{remark}{Remark}
\newcommand{\problem}[1]
  {\begin{center}\noindent\fbox{\begin{minipage}{\columnwidth}#1\end{minipage}}
  \end{center}}
\theoremstyle{remark}

\usepackage{etoolbox}
\makeatletter
\let\alignts@preamble\align@preamble
\patchcmd{\alignts@preamble}{\displaystyle}{\textstyle}{}{}
\patchcmd{\alignts@preamble}{\displaystyle}{\textstyle}{}{}

\def\alignts{\let\align@preamble\alignts@preamble\start@align\@ne\st@rredfalse\m@ne}

\makeatother

\allowdisplaybreaks

\title[Nonexistence of solitary waves]{Nonexistence of multi-dimensional solitary waves for the Euler-Poisson system}

\author[J. Bae]{Junsik Bae}
\address[JB]{Department of Mathematical Sciences, Ulsan National Institute of Science and Technology, Ulsan, 44919, Korea}
\email{junsikbae@unist.ac.kr}
 
\author[D. Kawagoe]{Daisuke Kawagoe}
\address[DK]{Graduate School of Informatics, Kyoto University, Yoshida-honmachi, Sakyo-ku, Kyoto 606-8501 Japan}
\email{d.kawagoe@acs.i.kyoto-u.ac.jp}

\date{\today} 


\begin{document}
 \begin{abstract} 
We study the nonexistence of multi-dimensional solitary waves for the Euler-Poisson system governing ion dynamics. It is well-known that the one-dimensional Euler-Poisson system has solitary waves that travel faster than the ion-sound speed. In contrast, we show that the two-dimensional and three-dimensional models do not admit nontrivial irrotational spatially localized traveling waves for any traveling velocity and for general pressure laws.  We derive some Pohozaev type identities associated with the energy and density integrals. This approach is extended to prove the nonexistence of irrotational multi-dimensional solitary waves for the  two-species Euler-Poisson system for ions and electrons.

\end{abstract}
\maketitle 

\section{Introduction}
We consider the multi-dimensional Euler-Poisson system for ion dynamics:
\begin{equation}\label{EP}
\left\{
\begin{array}{l l}
\partial_{t} \rho + \nabla \cdot (\rho \mathbf{u}) = 0, &  \\ 
\rho\left(\partial_{t} \mathbf{u}  + (\mathbf{u} \cdot \nabla ) \mathbf{u} \right) +   \nabla  p(\rho) = - \rho\nabla  \phi, &  (\mathbf{x}\in\mathbb{R}^n, \; t > 0, \; n =2,3),\\
-\Delta \phi = \rho - e^\phi,  &\\
p(\rho)=K\rho^\gamma.  &
\end{array} 
\right.
\end{equation}
 In this model, $\rho >0$, $\mathbf{u}\in \mathbb{R}^n$ and $\phi\in \mathbb{R}$ represent the ion number density, the fluid velocity vector field for ions, and  the electric potential, respectively. Especially, the electron number density is represented by $e^\phi$ (the \textit{Boltzmann relation}). The function $p(\rho)$ denotes the pressure of the ions, where $K\geq 0$ and $\gamma \geq 1$ are constants. The pressureless ($K=0$) Euler-Poisson system  is an ideal model for cold ions.

The dynamics of ion waves in an electrostatic plasma is often described using the model \eqref{EP}. This model is derived from the two-species Euler-Poisson system for ions and electrons based on the physical fact that the electron mass is much smaller than that of ions. We refer to \cite{Ch,Dav,Pecseli} for more physical background and \cite{GGPS} for a mathematical treatment of the zero mass limit. 

One of the significant phenomena in the dynamics of electrostatic plasma is the formation of solitary waves. Previous mathematical studies indicate that solitary waves dominate the long time dynamics of the one-dimensional Euler-Poisson system in a certain regime. Specifically, the existence of solitary waves  has been studied through phase plane analysis \cite{Cor, Sag}, and the linear stability of the solitary waves has been investigated employing the Evans function approach \cite{BK, HS}. Additionally, the works of \cite{BK2,Guo} rigorously justify the Korteweg-de Vries (KdV) approximation of the Euler-Poisson system. Furthermore, in a numerical perspective, the works of \cite{HNS,Satt} study solitary wave interactions in the Euler-Poisson system. 

In this paper, we are interested in the multi-dimensional behavior of the Euler-Poisson system. More specifically, we investigate  whether the Euler-Poisson system \eqref{EP} has nontrivial multi-dimensional solitary waves.\footnote{Such solutions also called lump solutions, or localized (in all spatial directions) solitary waves to distinguish them from line (one-dimensional) solitary waves in multi-dimensional settings.} In light of the works \cite{BS,BS3,KP,KS,Mi,RT}, this question is  also closely related to the transverse instability mechanism of line solitary waves in the  multi-dimensional Euler-Poisson system. Perturbing unstable line solitary waves often leads to the formation of localized solitary waves (see the numerical study of \cite{KS} for instance). To the best of our knowledge, the (in)stability of line solitary waves under the multi-dimensional Euler-Poisson flow has not been studied yet.

From the Euler-Poisson system \eqref{EP}, the Kadomtsev–Petviashvili (KP) equation 
\begin{equation}\label{KP}
\partial_x(\partial_t u + \partial_{xxx}u + u\partial_xu) =
\begin{cases}
 - \partial_{yy}u & \text{for } n=2, \\
 -\partial_{yy}u - \partial_{zz}u  &  \text{for } n=3,
\end{cases}
\end{equation}
is formally derived (see Appendix).\footnote{Other physical situations from which the three-dimensional KP-II equation is  derived does not seem to be known.} More precisely, the equation \eqref{KP} is called the KP-II equation. We refer to \cite{Pu} for its rigorous derivation from the two-dimensional Euler-Poisson system in a long wavelength regime. It is well-known  that, unlike the KP-I equation (\eqref{KP} with $\partial_{xxx}$ replaced by $-\partial_{xxx}$), the KP-II equation does not possess nontrivial localized traveling waves \cite{BM,BS}. Furthermore, while the solution to \eqref{EP} does not need to be irrotational flow for the KP-II approximation, the leading order term is irrotational (see Appendix).

\begin{remark}
We obtain the KP-II equation since the electric force is repulsive.  Indeed, the KP-I equation is derived from \eqref{EP} with $-\Delta\phi$ replaced by $\Delta\phi$ (see Appendix). However, the authors are not aware that in which physical context such a  system can be considered. 
\end{remark}

On the other hand, in the three-dimensional case, the irrotational smooth solution to \eqref{EP} exists globally in time for initial data that are sufficiently smooth and small, and additionally the amplitude of solution decays to 0 as $t \to +\infty$ \cite{GP}. As a rule, dispersion becomes stronger in higher dimensions, and the irrotationality of the flow somewhat weakens the nonlinear transport. Based on the above facts, one may expect the nonexistence of irrotational multi-dimensional traveling solitary waves, at least in a particular regime. 

Due to Galilean invariance, looking for a traveling wave of \eqref{EP} satisfying $\mbfu(\bm{\xi}) \to \bm{0}$ as $|\bm{\xi}| \to +\infty$, where $\bm{\xi}:=\mbfx +\mbfc t$, is equivalent to finding a stationary solution of \eqref{EP} satisfying $\mbfu(\mbfx) \to \mbfc$ as $|\mbfx| \to \infty$. Hence, we consider the system 
\begin{equation}\label{EP_S_S}
\left\{
\begin{array}{l l}
 \nabla\cdot(\rho \mathbf{u}) = 0, & \\
 \rho\left( (\mathbf{u} \cdot \nabla ) \mathbf{u} \right) +   \nabla  p(\rho) = - \rho\nabla  \phi, &  (\rho>0, \; \mathbf{u},\mathbf{x} \in \mathbb{R}^n, \; n=2,3), \\
 -\Delta \phi = \rho - e^\phi, & \\
 p(\rho)=K\rho^\gamma, & (K \geq 0, \; \gamma \geq 1),
\end{array}
\right.
\end{equation} 
with the far-field condition
\begin{equation}\label{BdCondi_S}
(\rho-1,\mathbf{u}-\mathbf{c},\phi) \to \mathbf{0} \quad \text{as } |\mathbf{x}| \to \infty.
\end{equation} 

For a given constant vector $\mathbf{c} \in \mathbb{R}^n$, a nontrivial classical solution to \eqref{EP_S_S}--\eqref{BdCondi_S} is called a \textit{solitary wave}.  

We prove the nonexistence of irrotational solitary waves for \eqref{EP_S_S}--\eqref{BdCondi_S}. It is well-known that $|\nabla \times \mbfu| = 0$ is preserved as long as solutions to \eqref{EP} are smooth.  Hence, looking for irrotational solutions to \eqref{EP_S_S}--\eqref{BdCondi_S} is not an overdetermined problem. We will further discuss the assumption of irrotational flows after presenting the main result of this paper. 
\begin{theorem}\label{MainThm}
For $n=2$ and $n=3$, consider the system \eqref{EP_S_S}--\eqref{BdCondi_S} with any given constant vector $\mbfc \in \R^n$. Then, \eqref{EP_S_S}--\eqref{BdCondi_S} does not admit any nontrivial classical solution satisfying $\rho-1, \mbfu-\mbfc, \phi \in L^1(\R^n)$ and $\partial_{x_j} u_k = \partial_{x_k} u_j$, where  $j,k = 1, \ldots,n$. 
\end{theorem}

We note that our nonexistence result holds for arbitrary velocity $\mbfc$. In contrast, the one-dimensional Euler-Poisson system for ions admits solitary waves that travel faster than the ion-sound speed  \cite{Cor,Sag}. Also, our result does not assume any smallness (or largeness) of wave amplitudes. 

We give a formal illustration that, in terms of the variational point of view, the irrotationality assumption in Theorem \ref{MainThm} is highly natural.  A traveling solitary wave with speed $c$ is usually a critical point of the conserved (in time $t$) quantity $\mathcal{H}-c\mathcal{M}$, where $\mathcal{H}$ is the Hamiltonian and $\mathcal{M}$ is the momentum associated with the translation invariance of the equation. For the Euler-Poisson system \eqref{EP} with the far-field condition $(\rho,\mbfu,\phi) \to (1,0,0)$ as $|x| \to \infty$, we have $\mathcal{H}=\mathcal{H}_n$ and $\mathcal{M}=\mathcal{M}_n$, where 
\begin{equation*} 
\begin{array}{l l}
 \mathcal{H}_n(\rho,\mbfu,\phi)  := \int_{\R^n} \frac{1}{2}\rho|\mbfu|^2 + \tilde{P}(\rho)  + \rho\phi - e^\phi + 1 - \frac{|\nabla\phi|^2}{2}\,d\mbfx,  \\
\mathcal{M}_n(\rho,\mbfu) :=  \int_{\R^n} \rho\mbfu\,d\mbfx,
\end{array} 
\end{equation*}
where $\tilde{P}$ is defined in \eqref{Ptil}. However, the solitary wave $(\rho_c,u_c,\phi_c)$ for the one-dimensional Euler-Poisson system with speed $c$ is a critical point of the functional $\mathcal{H}_1 - c \mathcal{I}_1$, rather than $\mathcal{H}_1 - c \mathcal{M}_1$, where  
\[
\mathcal{I}_n(\rho,\mbfu)  := \int_{\R^n} (\rho-1)\mbfu\,d\mbfx
\]
is the momentum of ``excess mass''. This is due to the fact that $\rho_c \to 1 \neq 0$ as  $|\mbfx|\to \infty$.\footnote{A similar situation occurs in the other dispersive equation with a non-zero background. See \cite{BSa,BGS} for instance.}  On the other hand, for $n=2,3$, $\mathcal{I}_n$ is not conserved unless $|\mbfu \times (\nabla \times \mbfu)|=0$. Indeed, the following identity holds:
\begin{equation}\label{VecIden}
\frac{1}{2}\nabla|\mbfu|^2 = (\mbfu \cdot \nabla) \mbfu + \mbfu \times (\nabla \times \mbfu).
\end{equation}

Moreover, if $(\rho_\mbfc,\mbfu_\mbfc,\phi_\mbfc)$ is a critical point of $\mathcal{H}_n - \mbfc \cdot\mathcal{I}_n$, then we have
\[
\frac{\delta (\mathcal{H}_n - \mbfc \cdot\mathcal{I}_n)}{\delta \rho} (\rho_\mbfc,\mbfu_\mbfc,\phi_\mbfc)  = \frac{1}{2}|\mbfu_\mbfc|^2 + \tilde{P}'(\rho_\mbfc) + \phi_\mbfc - \mbfc \cdot \mbfu_\mbfc =0
\]
(i.e. Bernoulli’s Theorem), which holds when $|(\mbfu_\mbfc-\mbfc) \times (\nabla \times \mbfu_\mbfc)|=0$ due to \eqref{VecIden}.

We remark that the solutions considered in Theorem \ref{EP} are traveling solitary waves of \eqref{EP} with finite energy, that is, $\mathcal{H}_n < \infty$ (see the comment below Lemma \ref{Lem2}). It is obvious that the line solitary waves of \eqref{EP} are irrotational. However, for these waves, $\mathcal{H}_n=\infty$ since they are localized only in one direction.

We outline the proof of Theorem \ref{MainThm}. We prove the nonexistence of solitary waves by contradiction using some Pohozaev type identities. Since the flow is irrotational, $\phi$ is explicitly expressed in terms of other unknown functions (see \eqref{EP_S3}), and we derive 
\begin{subequations}\label{En}
\begin{align}
& 0  = \frac{1}{2}\int_{\R^n} \rho(|\mbfu|^2 - |\mbfc|^2) \,d\mbfx + \int_{\R^n} \tilde{P}(\rho)\,d\mbfx + \int_{\R^n} (\rho \phi - e^\phi + 1)  - \frac{n - 2}{2n}  |\nabla \phi|^2\,d\mbfx ,  \label{En_1}\\
& 0  = \frac{1}{2}\int_{\R^n} \rho(|\mbfu|^2-|\mbfc|^2) \,d\mbfx + \int_{\R^n} \tilde{P}(\rho)\,d\mbfx + \int_{\R^n} \tilde{Q}(\phi) \,d\mbfx, \label{En_2}
\end{align}
\end{subequations}
where 
\begin{subequations}\label{PQtil}
\begin{align}
& \tilde{P}(\rho) :=
\begin{cases}
K (\rho \ln \rho - \rho + 1), &\gamma = 1, \\
\frac{K \rho}{\gamma - 1} (\rho^{\gamma-1} - 1) - K(\rho - 1), &\gamma > 1, 
\end{cases} \label{Ptil} \\
& \tilde{Q}(\phi) := e^\phi \phi - e^\phi + 1 +  \frac{n + 2}{2n} |\nabla \phi|^2. \label{Qtil}
\end{align}
\end{subequations}
In \eqref{En}, the integrals correspond to the kinetic energy, the pressure potential energy, and the electric potential energy, respectively. The first form of the electric potential energy naturally arises from the Hamiltonian structure of \eqref{EP}, while it is clear that the potential energy $\tilde{Q}$ in the second form is non-negative. Using the non-negativity of $\tilde{P}$ and $\tilde{Q}$, we show that the kinetic energy part of \eqref{En_2} must be strictly negative if $(\rho,\mbfu,\phi)$ is nontrivial. We remark that $\tilde{P}(\rho)$ and $\tilde{Q}(\phi)$ have the same sign because the pressure and electric forces are both repulsive. 

On the other hand, we obtain an identity relating the kinetic energy integral to the mass integral (see \eqref{Ident_4}). This identity, together with the \textit{quasi-neutrality} (see \eqref{Ident_9}), allows us to deduce from \eqref{En_1} that the kinetic energy is non-negative, which is a contradiction. We justify the above argument in Section 2. 

The aforementioned argument applies to the nonexistence of the multi-dimensional solitary waves for the two-species Euler-Poisson system for ions and electrons, with a slight modification due to the different form of the Poisson equation. It is worth mentioning that in the small electron mass regime, traveling solitary waves for the one-dimensional two-species Euler-Poisson system exist   \cite{Cor}. We discuss the two-species model in Section 3.

We conclude this section with some open questions. Our proof strongly relies on the assumption that the flow is irrotational and that solitary waves subtracted by some constants belong to the $L^1$ space. It would be interesting to study (non)existence of solitary waves in a more general setting.

Additionally, like the other dispersive equations, Theorem \ref{MainThm} motivates the scattering problem (the long-term behavior resembles the linear dynamics) for irrotational smooth solutions to the multi-dimensional Euler-Poisson system in the $L^1$ space for small initial data. It seems unlikely to expect the scattering phenomenon in the one-dimensional case,  since solitary waves exist for arbitrarily small amplitude and decay exponentially. Moreover, they are linearly stable.

\section{Proof of Theorem \ref{MainThm}}

Let $\mbfc$ be a constant vector in $\R^n$, and we consider the following system:
\begin{subequations}\label{EP_S}
\begin{align}[left = \empheqlbrace\,] 
& \nabla \cdot (\rho \mbfu) = 0 , \label{EP_S1} \\
& \phi = \frac{|\mbfc|^2}{2} -  \frac{|\mbfu|^2}{2} - P(\rho), \label{EP_S3} \\
& \rho = e^\phi - \Delta \phi, \label{EP_S2} \\
& \partial_{x_j} u_k = \partial_{x_k} u_j, \qquad j,k = 1, \ldots,n, \label{EP_S4}\\
& (\rho, \mbfu, \phi) \to (1, \mbfc, 0) \text{ as } |\mbfx| \to +\infty,  \qquad (\rho>0, \; \mathbf{u},\mathbf{x} \in \mathbb{R}^n, \; n=2,3),  \label{EP_S5}
\end{align}
\end{subequations}
where $x_i$ and $u_i$ denote the $i$-th component of $\mbfx$ and $\mbfu$, respectively, and 
\begin{equation}\label{P}
P(\rho) := 
\begin{cases}
K \ln \rho, &\gamma = 1, \\
\frac{K \gamma}{\gamma - 1} (\rho^{\gamma-1} - 1), &\gamma > 1, 
\end{cases}
\end{equation}
with $K \geq 0$.   Here, the identity \eqref{EP_S3} follows from the second equation of \eqref{EP_S_S} and \eqref{BdCondi_S} thanks to the vector calculus identity \eqref{VecIden} and \eqref{EP_S4}.

We first prove some preliminary lemmas. Let $\mu_0(s)$ be a smooth cut-off function on $\R_{\geq 0}$ satisfying
\[
\mu_0(s) = 1 \quad \text{for } 0 \leq s \leq 1, \quad \mu_0(s) = 0 \quad \text{for }  s \geq 2.
\]
We set $\mu_r(\mbfx) = \mu_0(|\mbfx|^2/r^2)$ for all positive integers $r$. We remark that
\[
\lim_{r \to \infty} \int_{\R^n} f(\mbfx) \mu_r(\mbfx)\,d\mbfx = \int_{\R^n} f(\mbfx)\,d\mbfx
\]
for $f \in L^1(\R^n)$ by the Lebesgue dominated convergence theorem.

\begin{lemma} \label{lem:derivative_vanish}
For $f \in L^1(\R^n)$, we have
\[
\lim_{r \to \infty} \int_{\R^n} |f(\mbfx)| |\mbfx| |\nabla \mu_r(\mbfx)|\,d\mbfx = 0.
\]
\end{lemma}

\begin{proof}
We observe that
\[
\nabla \mu_r(\mbfx) = \frac{2}{r^2} \mu_0' \left(\frac{|\mbfx|^2}{r^2} \right) \mbfx,
\]
and that $\supp \nabla \mu_r \subset \{ \mbfx \in \R^n \mid r^2 \leq |\mbfx|^2 \leq 2 r^2 \}$. Thus, we have
\[
|f(\mbfx)| |\mbfx| |\nabla \mu_r(\mbfx)| \leq 4 \| \mu_0' \|_{C_b[0, \infty)} |f(\mbfx)|
\]
for a.e. $\mbfx \in \R^n$, where the right hand side is integrable. Here, $\| \cdot \|_{C_b[0, \infty)}$ denotes the maximum norm for bounded continuous functions on $[0, \infty)$. The conclusion follows from the Lebesgue dominated convergence theorem. 
\end{proof}

\begin{lemma}\label{Lem2}
Suppose that $(\rho, \mbfu, \phi)$ is a classical solution of the system \eqref{EP_S}--\eqref{P} satisfying $\rho - 1, \mbfu - \mbfc, \phi \in L^1(\R^n)$. Then, the following identities hold:  
\begin{align}
&\int_{\R^n} (\mbfu - \mbfc)\,d\mbfx =\mathbf{0}, \label{Ident_2}\\
&\int_{\R^n} (\rho \mbfu - \mbfc) \cdot \mbfu \,d\mbfx = 0. \label{Ident_3}\\
&\int_{\R^n} \rho (|\mbfu|^2 - |\mbfc|^2 ) \,d\mbfx = -|\mbfc|^2 \int_{\R^n} (\rho - 1)\,d\mbfx, \label{Ident_4}\\
&\int_{\R^n} \rho \phi \,d\mbfx = \int_{\R^n} e^\phi \phi \,d\mbfx + \int_{\R^n} |\nabla \phi|^2\,d\mbfx, \label{Ident_8}\\
&\int_{\R^n} (e^\phi - 1) \,d\mbfx = \int_{\R^n} (\rho - 1)\,d\mbfx, \label{Ident_9} \\
&\lim_{r \to \infty} \int_{\R^n} (\mbfx \cdot \nabla \rho) \mu_r \,d\mbfx = -n \int_{\R^n} (\rho - 1)\,d\mbfx, \label{Ident_5}\\
&\lim_{r \to \infty} \int_{\R^n} (\mbfx \cdot \nabla \rho) P(\rho) \mu_r \,d\mbfx = -n \int_{\R^n} \tilde{P}(\rho)\,d\mbfx, \label{Ident_6}\\
&\lim_{r \to \infty} \int_{\R^n} (\mbfx \cdot \nabla \rho) \phi \mu_r \,d\mbfx = - n \int_{\R^n} (\rho \phi -e^\phi + 1) \,d\mbfx + \frac{n - 2}{2} \int_{\R^n} |\nabla \phi|^2\,d\mbfx, \label{Ident_7}
\end{align} 
where $\tilde{P}$ is defined in \eqref{Ptil}.
\end{lemma}

Before proving Lemma \ref{Lem2}, we remark integrability of functions in \eqref{Ident_2}--\eqref{Ident_7} under the assumption of Lemma \ref{Lem2}. By continuity and \eqref{EP_S5}, the solution $(\rho,\mbfu,\phi)$ satisfies
\begin{equation}\label{Bdness}
0 < \inf_{\mbfx \in \R^n} \rho  < \sup_{\mbfx \in \R^n} \rho < \infty,  \quad \rho, \mbfu, \phi \in C_b(\R^n).
\end{equation}
Since 
\[
|P(\rho)| = K \left| \int_1^\rho r^{\gamma-2}\,dr \right| \leq K \max \left\{ \left( \inf_{\mbfx \in \R^n} \rho \right)^{\gamma-2}, \left( \sup_{\mbfx \in \R^n} \rho \right)^{\gamma-2} \right\} |\rho - 1|
\]
and
\begin{equation} \label{tildeP}
\tilde{P}(\rho) = \frac{\rho}{\gamma} P(\rho) - K(\rho - 1),
\end{equation}
we have $P(\rho)$, $ \tilde{P}(\rho) \in L^1(\R^n) \cap C_b(\R^n)$ from $\rho-1 \in L^1(\R^n) \cap C_b(\R^n)$. The identity \eqref{tildeP} will be used later again. 

Also, since 
\[
 |e^\phi - 1| = \left| \int_0^\phi e^s\,ds \right| \leq e^{\| \phi \|_{C_b(\R^n)}} |\phi|, 
 \]
we have $e^\phi - 1 \in L^1(\R^n) \cap C_b(\R^n)$ from $\phi \in L^1(\R^n) \cap C_b(\R^n)$, and hence we see that $\Delta \phi = e^\phi - 1 - (\rho - 1) \in L^1(\R^n) \cap C_b(\R^n)$. In addition, we have $\nabla \phi \in L^2(\R^n)$. Indeed, by integration by parts, we have
\[
\int_{\R^n} |\nabla \varphi|^2\,d\mbfx = - \int_{\R^n} (\Delta \varphi) \varphi\,d\mbfx \leq \frac{1}{2} \left( \| \Delta \varphi \|_{L^2(\R^n)}^2 + \| \varphi \|_{L^2(\R^n)}^2 \right)
\]
for all $\varphi \in C^\infty_0(\R^n)$ and hence for all $\varphi \in L^2(\R^n)$ with $\Delta \varphi \in L^2(\R^n)$ by completion. Thus, we have $\nabla \phi \in L^2(\R^n)$ since $\phi, \Delta\phi \in L^1(\R^n) \cap C_b(\R^n) \subset L^2(\R^n)$.  

The integrability of other functions are straightforward to check.  

\begin{proof}[Proof of Lemma \ref{Lem2}]

For the first identity \eqref{Ident_2}, we make use of the irrotationality  \eqref{EP_S4}. For $j \neq k$, we have
\begin{align*}
0 =& \int_{\R^n} \left( \partial_{x_j}(u_k - c_k) - \partial_{x_k}(u_j - c_j) \right) x_k \mu_r\,d\mbfx\\
=& - \int_{\R^n} (u_k - c_k) x_k \partial_{x_j} \mu_r\,d\mbfx + \int_{\R^n} (u_j - c_j) x_k \partial_{x_k} \mu_r\,d\mbfx + \int_{\R^n} (u_j - c_j) \mu_r\,d\mbfx,
\end{align*}
and letting $r \to \infty$, we obtain
\[
\lim_{r \to \infty} \int_{\R^n} (u_j - c_j) \mu_r\,d\mbfx = 0, \quad  j = 1, \ldots, n
\]
by Lemma \ref{lem:derivative_vanish}. Since $\mbfu -\mbfc \in L^1(\R^n)$, we have \eqref{Ident_2} by the Lebesgue dominated convergence theorem. 

To prove \eqref{Ident_3}, we introduce the potential function $\psi$ such that $\mbfu = \nabla \psi$, whose existence is guaranteed by \eqref{EP_S4}. Without loss of generality, we set $\psi(\mathbf{0}) = 0$. We multiply \eqref{EP_S1} by $\psi \mu_r$ and integrate it. Then, we see that
\begin{equation*}
\begin{split}
0 
& = \int_{\R^n} \nabla \cdot (\rho \mbfu - \mbfc) \psi \mu_r\,d\mbfx\\ 
& = - \int_{\R^n} (\rho \mbfu - \mbfc) \cdot \mbfu \mu_r\,d\mbfx - \int_{\R^n} \big((\rho \mbfu - \mbfc) \cdot \nabla \mu_r \big)\psi\,d\mbfx,
\end{split}
\end{equation*}
from which \eqref{Ident_3} follows by noting that  
\[
| (\rho \mbfu - \mbfc) \psi(\mbfx)| = |\rho \mbfu - \mbfc|\left| \int_0^1 \frac{d}{ds} \psi(s \mbfx) \,ds \right| \leq  |\rho \mbfu - \mbfc||\mbfx| \| \mbfu \|_{L^\infty(\R^n)}
\]
and by applying Lemma \ref{lem:derivative_vanish} since $\rho \mbfu - \mbfc = (\rho-1)\mbfu +\mbfu - \mbfc \in L^1(\R^n)$.

 The identity \eqref{Ident_4} follows from \eqref{Ident_2} and \eqref{Ident_3} as follows:
\[
\begin{split}
\int_{\R^n} \rho(|\mbfu|^2-|\mbfc|^2)\,d\mbfx + \int_{\R^n} |\mbfc|^2(\rho-1)\,d\mbfx 
& =  \int_{\R^n} \rho|\mbfu|^2-|\mbfc|^2\,d\mbfx \\
& =  \int_{\R^n} (\rho \mbfu - \mbfc) \cdot \mbfu\,d\mbfx +  \int_{\R^n} \mbfc\cdot(\mbfu - \mbfc)\,d\mbfx  \\
& = 0.
\end{split}
\]

The identity \eqref{Ident_8} follows from \eqref{EP_S2}:
\begin{align*}
\int_{\R^n} \rho \phi \mu_r\,d\mbfx =& \int_{\R^n} (e^\phi - \Delta \phi) \phi \mu_r\,d\mbfx\\
=& \int_{\R^n} e^\phi \phi \mu_r\,d\mbfx + \int_{\R^n} |\nabla \phi|^2 \mu_r \,d\mbfx + \int_{\R^n} \phi \nabla \phi \cdot \nabla \mu_r\,d\mbfx,
\end{align*}
where the last term vanishes as $r \to \infty$ since 
\[
|\phi \nabla \phi \cdot \nabla \mu_r| \leq |\phi| | \nabla \phi |\frac{\sqrt{2} \| \mu_0' \|_{C_b[0, \infty)}}{r} \in L^1(\R^n).
\] 

To show the identity \eqref{Ident_9}, we use \eqref{EP_S2} again to obtain 
\begin{equation}\label{4} 
\int_{\R^n} (e^\phi - 1) \mu_r\,d\mbfx - \int_{\R^n} (\rho - 1) \mu_r\,d\mbfx = -\int_{\R^n} \Delta \phi \mu_r\,d\mbfx = - \int_{\R^n} \phi \Delta \mu_r\,d\mbfx.
\end{equation}
Since 
\[
\Delta \mu_r = \frac{4 |\mbfx|^2}{r^4} \mu_0'' \left( \frac{|\mbfx|^2}{r^2} \right) + \frac{2 n}{r^2} \mu_0' \left( \frac{|\mbfx|^2}{r^2} \right),
\]
we have
\[
|\Delta \mu_r| \leq \frac{8 \| \mu_0'' \|_{C_b[0, \infty)}}{r^2} + \frac{2 n \| \mu_0' \|_{C_b[0, \infty)}}{r^2}.
\]
Thus, the identity \eqref{Ident_9} is obtained from \eqref{4} by taking $r \to \infty$ since $e^\phi-1, \rho-1 \in L^1(\R^n)$.

For the identity \eqref{Ident_5}, we perform integration by parts to get
\begin{align*}
\int_{\R^n} (\mbfx \cdot \nabla \rho) \mu_r \,d\mbfx =& \int_{\R^n} (\mbfx \cdot \nabla (\rho - 1)) \mu_r \,d\mbfx\\
=& - n \int_{\R^n}  (\rho - 1) \mu_r \,d\mbfx - \int_{\R^n} (\rho - 1) (\mbfx \cdot \nabla \mu_r) \,d\mbfx,
\end{align*}
from which the identity \eqref{Ident_5} is obtained by Lemma \ref{lem:derivative_vanish} since $\rho - 1 \in L^1(\R^n)$. The identity \eqref{Ident_6} can be proved in the same way noting that $ P(\rho)\nabla \rho = \nabla \tilde{P}(\rho)$ and $\tilde{P}(\rho) \in L^1(\R^n)$.

Lastly, we prove \eqref{Ident_7}. Integrating by parts and using \eqref{EP_S2}, we have that
\[
\begin{split}
\int_{\R^n} (\mbfx \cdot \nabla \rho) \phi \mu_r \,d\mbfx
&  = -n\int_{\R^n} \rho \phi \mu_r \,d\mbfx  - \int_{\R^n} \rho (\nabla\phi\cdot\mbfx) \mu_r \,d\mbfx - \int_{\R^n} \rho \phi \mbfx\cdot\nabla\mu_r \,d\mbfx \\
& =  -n\int_{\R^n} \rho \phi \mu_r \,d\mbfx  + \int_{\R^n} (\Delta\phi - e^\phi) (\nabla\phi\cdot\mbfx) \mu_r \,d\mbfx - \int_{\R^n} \rho \phi \mbfx\cdot\nabla\mu_r \,d\mbfx  \\
& =  -n\int_{\R^n} (\rho \phi-e^\phi + 1) \mu_r \,d\mbfx  + \int_{\R^n} \Delta\phi (\nabla\phi\cdot\mbfx) \mu_r \,d\mbfx \\
& \quad - \int_{\R^n} (\rho \phi-e^\phi + 1) \mbfx\cdot\nabla\mu_r\,d\mbfx.
\end{split}
\]
Here, since $\rho\phi-e^\phi+1 \in L^1(\R^n)$, we only consider the second term. Integrating by parts twice, it is straightforward to see that 
\[
\begin{split}
\int_{\R^n} \Delta\phi (\nabla\phi\cdot\mbfx) \mu_r \,d\mbfx
 =& \frac{n-2}{2}\int_{\R^n} |\nabla\phi|^2\mu_r \,d\mbfx +\frac{1}{2} \int_{\R^n} |\nabla \phi|^2 \mbfx \cdot \nabla \mu_r\,d\mbfx\\
& - \int_{\R^n} (\mbfx \cdot \nabla \phi) (\nabla \phi \cdot \nabla \mu_r)\,d\mbfx.
\end{split}
\]
Here, the last two integrals vanish as $r \to \infty$ since $\nabla\phi \in L^2(\R^n)$. Hence, we obtain \eqref{Ident_7}.

We complete the proof.
\end{proof}

\begin{proof}[Proof of Theorem \ref{MainThm}]
Suppose that $(\rho,\mbfu,\phi)$ is a classical solution of the system \eqref{EP_S} satisfying $\rho-1,\mbfu-\mbfc,\phi \in L^1(\R^n)$.  We show \eqref{En_1} and \eqref{En_2}, and then derive a contradiction.

To show \eqref{En_1},  we first derive the following identity:
\begin{equation} \label{2}
0 = - \lim_{r \to \infty} \frac{1}{2} \int_{\R^n} (\mbfx \cdot \nabla \rho) (|\mbfu|^2 -|\mbfc|^2) \mu_r\,d\mbfx - \frac{n}{2} \int_{\R^n} \rho (|\mbfu|^2 - |\mbfc|^2)\,d\mbfx.
\end{equation}
From \eqref{EP_S1} and integration by parts, we have
\begin{align*}
0 =& \int_{\R^n} \nabla \cdot (\rho \mbfu - \mbfc) (\mbfx \cdot \mbfu) \mu_r\,d\mbfx\\
=& - \sum_{j,k = 1}^n \int_{\R^n} (\rho u_j - c_j) x_k \partial_{x_j} u_k \mu_r\,d\mbfx \\ 
&  - \int_{\R^n} (\rho \mbfu - \mbfc) \cdot \mbfu \mu_r\,d\mbfx  -\int_{\R^n} (\mbfx \cdot \mbfu) (\rho \mbfu - \mbfc) \cdot \nabla \mu_r\,d\mbfx. 
\end{align*}
Here, the last two integrals vanish as $r \to \infty$ by the identity  \eqref{Ident_3} and Lemma \ref{lem:derivative_vanish}. Thus, in what follows, we only discuss the first term.

By the irrotationality \eqref{EP_S4}, we observe that
\begin{align*}
\sum_{j,k = 1}^n \int_{\R^n} (\rho u_j - c_j) x_k \partial_{x_j} u_k \mu_r\,d\mbfx =& \sum_{j,k = 1}^n \int_{\R^n} (\rho u_j - c_j) x_k \partial_{x_k} u_j \mu_r\,d\mbfx\\
=&\frac{1}{2} \int_{\R^n} \rho \mbfx \cdot \nabla |\mbfu|^2 \mu_r\,d\mbfx - \int_{\R^n} (\mbfx \cdot \nabla (\mbfc \cdot \mbfu) ) \mu_r\,d\mbfx.
\end{align*}
Here, the latter part vanishes as $r \to \infty$ thanks to \eqref{Ident_2} and Lemma \ref{lem:derivative_vanish} since we have 
\[
\int_{\R^n} (\mbfx \cdot \nabla (\mbfc \cdot \mbfu) ) \mu_r\,d\mbfx = - n \mbfc \cdot \int_{\R^n} (\mbfu - \mbfc) \mu_r\,d\mbfx - \mbfc \cdot \int_{\R^n} (\mbfu - \mbfc) (\mbfx \cdot \nabla \mu_r)\,d\mbfx.
\]
 On the other hand, for the former part, we have
\begin{align*}
\frac{1}{2} \int_{\R^n} \rho \mbfx \cdot \nabla |\mbfu|^2 \mu_r\,d\mbfx =& - \frac{1}{2} \int_{\R^n} (\mbfx \cdot \nabla \rho) (|\mbfu|^2 -|\mbfc|^2) \mu_r\,d\mbfx - \frac{n}{2} \int_{\R^n} \rho (|\mbfu|^2 - |\mbfc|^2)\mu_r\,d\mbfx\\
&-\frac{1}{2} \int_{\R^n} \rho (|\mbfu|^2 - |\mbfc|^2) \mbfx \cdot \nabla \mu_r\,d\mbfx,
\end{align*}
where the last term vanishes as $r \to \infty$ by Lemma \ref{lem:derivative_vanish}. By summarizing the above argument, we obtain \eqref{2}.

On the other hand, from \eqref{EP_S3}, \eqref{Ident_6} and \eqref{Ident_7}, we have
\begin{equation}\label{3}
\begin{split}
&\frac{1}{2} \lim_{r \to \infty} \int_{\R^n} (\mbfx \cdot \nabla \rho) (|\mbfu|^2 - |\mbfc|^2) \mu_r\,d\mbfx\\ 
=& - \lim_{r \to \infty} \int_{\R^n} (\mbfx \cdot \nabla \rho) (\phi + P(\rho)) \mu_r\,d\mbfx\\
=& n \int_{\R^n} (\rho \phi - e^\phi + 1)\,d\mbfx - \frac{n - 2}{2} \int_{\R^n} |\nabla \phi|^2\,d\mbfx + n \int_{\R^n} \tilde{P}(\rho)\,d\mbfx.
\end{split}
\end{equation}
Combining \eqref{2} and \eqref{3}, we obtain \eqref{En_1}. Moreover, from  \eqref{En_1} and \eqref{Ident_8}, we have \eqref{En_2}. 

Now, we substitute \eqref{Ident_4} into \eqref{En_2} to get
\begin{equation}\label{En_A2}
0  = -\frac{|\mbfc|^2}{2} \int_{\R^n} (\rho - 1)\,d\mbfx + \int_{\R^n} \tilde{P}(\rho)\,d\mbfx + \int_{\R^n} \tilde{Q}(\phi) \,d\mbfx.
\end{equation}
We note that $\tilde{P}(\rho) \geq 0$ and $\tilde{Q}(\phi) \geq 0$, and they vanish only at $\rho=1$ and $\phi=0$, respectively. Moreover, if $\rho=1$, then $\mbfu=\mbfc$ since
\begin{equation}\label{5}
\int_{\R^n} |\mbfu-\mbfc|^2  \,d\mbfx= \int_{\R^n} (\mbfu - \mbfc) \cdot \mbfu \,d\mbfx - \int_{\R^n} (\mbfu - \mbfc) \cdot \mbfc \,d\mbfx = 0
\end{equation}
by \eqref{Ident_2} and \eqref{Ident_3}.
Hence, we conclude that  for $\mbfc=\mathbf{0}$, $(\rho,\mbfu,\phi)$ must be trivial, and for $|\mbfc| \neq 0$, if $(\rho,\mbfu,\phi)$ is nontrivial, then we must have 
\begin{equation*} 
\int_{\R^n} (\rho-1) \,d\mbfx > 0.
\end{equation*}
In what follows, we derive  $\int_{\R^n} (\rho-1) \,d\mbfx \leq 0$,
contrary to the above strict inequality.

Substituting \eqref{EP_S3} and \eqref{Ident_9} into  \eqref{En_1}, we have 
\begin{equation*} 
\begin{split}
0 
&  = \frac{1}{2}\int_{\R^n}\rho ( |\mbfu|^2 -|\mbfc|^2 ) \,d\mbfx + \int_{\R^n}\tilde{P}(\rho)\,d\mbfx \\
& \quad + \int_{\R^n} \rho \left(\frac{|\mbfc|^2}{2} - \frac{|\mbfu|^2}{2} - P(\rho) \right) - (\rho - 1) \,d\mbfx- \frac{n - 2}{2n} \int_{\R^n} |\nabla \phi|^2\,d\mbfx \\
& =  \int_{\R^n} \tilde{P}(\rho) - \rho P(\rho)\,d\mbfx  - \int_{\R^n} (\rho - 1) \,d\mbfx - \frac{n - 2}{2n} \int_{\R^n} |\nabla \phi|^2\,d\mbfx.
\end{split}
\end{equation*} 
From \eqref{tildeP}, we observe 
\begin{equation}\label{6}
\tilde{P}(\rho) - \rho P(\rho) = (1 - \gamma) \tilde{P}(\rho) - K \gamma (\rho - 1).
\end{equation}
Hence, since $\tilde{P} \geq 0$ and $1- \gamma \leq 0$, we have
\begin{equation*} 
\begin{split}
0 
& = (1 - \gamma) \int_{\R^n} \tilde{P}(\rho)\,d\mbfx - (K \gamma + 1) \int_{\R^n} (\rho-1) \,d\mbfx - \frac{n - 2}{2n} \int_{\R^n} |\nabla \phi|^2\,d\mbfx \\
& \leq - (K \gamma + 1) \int_{\R^n}(\rho-1) \,d\mbfx,
\end{split}
\end{equation*}
or since $K\gamma +1 >0$, 
\[
\int_{\R^n} (\rho - 1)\,d\mbfx \leq 0,
\]
which is a contradiction. This completes the proof.
\end{proof}

\section{Nonexistence for the Two-species Euler-Poisson system} 
The strategy of the proof of Theorem \ref{MainThm} also applies to the nonexistence of solitary waves for the two-species Euler-Poisson system. In this section, we consider the following irrotational steady two-species Euler-Poisson system:
\begin{subequations}\label{2EP_S}
\begin{align}[left = \empheqlbrace\,] 
& \nabla \cdot (\rho_\alpha \mbfu_\alpha) = 0 , \label{2EP_Sp} \\
& \kappa_\alpha\phi = \frac{\eta_\alpha}{2}(|\mbfc_\alpha|^2 -  |\mbfu_\alpha|^2) - P_\alpha(\rho_\alpha),  \label{2EP_Sp1}  \\
& -\Delta \phi = \rho_i - \rho_e,   \label{2EP_Sp2} \\
& \partial_{x_j} u_{\alpha, k} = \partial_{x_k} u_{\alpha, j}, \qquad j,k = 1, \ldots,n, \label{2EP_S4}\\ 
& (\rho_\alpha,  \mbfu_\alpha,  \phi) \to (1, \mbfc_\alpha, 0) \text{ as } |\mbfx| \to +\infty,   \qquad (\rho_\alpha > 0, \;\mbfu_\alpha, \mathbf{x} \in \mathbb{R}^n, \; n=2,3), \label{2EP_S5}
\end{align}
\end{subequations}
where $\alpha$ is the index for the ions ($\alpha = i$) and the electrons ($\alpha= e$); $\rho_\alpha$ and $\mbfu_\alpha $ represent the density and the fluid velocity field of each species;  $\phi$ denotes the electric potential;
\begin{equation}\label{P_al}
P_\alpha(\rho) := 
\begin{cases}
K_\alpha \ln \rho_\alpha, &\gamma_\alpha = 1, \\
\frac{K_\alpha \gamma_\alpha}{\gamma_\alpha - 1} (\rho_\alpha^{\gamma_\alpha-1} - 1), &\gamma_\alpha > 1, 
\end{cases}
\end{equation}
with $K_\alpha \geq 0$, and
\begin{align}\label{kapa}
& \kappa_\alpha :=
\begin{cases}
1, & \alpha = i, \\
-1, & \alpha = e,
\end{cases} 
\qquad \eta_\alpha :=
\begin{cases}
1, & \alpha = i, \\
\geq 0, & \alpha = e.
\end{cases} 
\end{align}
In \eqref{kapa},  $\eta_e \geq 0$ is a constant for the ratio of the electron mass to the ion mass. In fact, the Boltzmann relation $\rho_e = e^\phi$ in \eqref{EP} is obtained from \eqref{2EP_Sp1} by letting $(\eta_e,\gamma_e) = (0,1)$ and $K_e>0$ (zero-mass isothermal electrons). In the context of plasma physics, $\eta_e$ is very small. Moreover, due to the high mobility of electrons, the isothermal pressure is often assumed for electrons \cite{Ch}.

\begin{theorem}\label{MainThm2}
For $n=2$ and $n=3$, consider the system \eqref{2EP_S}--\eqref{kapa} with any given constant vector $(\mbfc_i,\mbfc_e) \in \R^n\times\R^n$. Then, for $(K_i, K_e) \neq (0,0)$, \eqref{2EP_S}--\eqref{kapa} does not admit any nontrivial classical solution satisfying $\rho_\alpha-1, \mbfu_\alpha-\mbfc_\alpha, \phi \in L^1(\R^n)$, where $\alpha=i,e$.
\end{theorem}
 
 We remark that our strategy does not work for the case $K_i=K_e=0$, and in this case, the one-dimensional model does not admit traveling solitary waves.\footnote{When $K_i=K_e=0$, it is straightforward to see that the one-dimensional system \eqref{2EP_Sp}--\eqref{2EP_Sp2} with \eqref{2EP_S5} is reduced to the following second order ODE:
\[
\phi''=\sqrt{\frac{c_e^2}{c_e^2+2\phi}} - \sqrt{\frac{c_i^2}{c_i^2-2\phi}},
\]
where $c_i,c_e \neq 0$. The stationary point $(\phi,\phi')=(0,0)$ is a center point, not a saddle point. If $c_i=c_e=0$, then the solution must be trivial. For the other pressure laws, we refer to \cite{Cor} for the existence results.
}

To prove Theorem \ref{MainThm2}, we define for $\alpha = i,e$,
\[
\tilde{P}_\alpha(\rho):=
\begin{cases}
K_\alpha (\rho \ln \rho - \rho + 1), &\gamma_\alpha = 1,  \\
\frac{K_\alpha \rho}{\gamma_\alpha - 1} (\rho^{\gamma_\alpha-1} - 1) - K_\alpha(\rho - 1), &\gamma_\alpha > 1, 
\end{cases}
\]
with $K_\alpha \geq 0$. We note that $\tilde{P}_\alpha$ corresponds to \eqref{Ptil}. The following lemma can be proved in much the same way as Lemma \ref{Lem2}. We omit the details.
\begin{lemma}\label{Lem3}
Suppose that $(\rho_i, \mbfu_i, \rho_e, \mbfu_e, \phi)$ is a classical solution of the system \eqref{2EP_S}--\eqref{kapa} satisfying $\rho_\alpha - 1, \mbfu_\alpha - \mbfc_\alpha, \phi \in L^1(\R^n)$ for $\alpha=i,e$. Then, the identities in Lemma \ref{Lem2} with $(\rho,\mbf{u},P,\tilde{P})$ replaced by $(\rho_\alpha,\mbf{u}_\alpha,P_\alpha,\tilde{P}_\alpha)$ hold except \eqref{Ident_8}, \eqref{Ident_9}, and \eqref{Ident_7}; it holds that
\begin{align}
& \int_{\R^n} ( \rho_i-\rho_e)\phi\,d\mbfx = \int_{\R^n}|\nabla\phi|^2\,d\mbfx, \label{2sp3}\\
&  \int_{\R^n} (\rho_i - 1) \,d\mbfx = \int_{\R^n} (\rho_e - 1) \,d\mbfx, \label{2sp4}\\
& \sum_{\alpha=i,e}\lim_{r \to \infty}\int_{\R^n} (\mbfx \cdot \nabla \rho_\alpha)\kappa_\alpha \phi \mu_r\,d\mbfx = -n\int_{\R^n}(\rho_i - \rho_e)\phi\,d\mbfx + \frac{n-2}{2}\int_{\R^n}|\nabla\phi|^2\,d\mbfx.\label{2sp5}
\end{align}  
\end{lemma}

The identities \eqref{2sp3}, \eqref{2sp4}, and \eqref{2sp5} in Lemma \ref{Lem3} correspond to the identities \eqref{Ident_8}, \eqref{Ident_9}, and \eqref{Ident_7} in Lemma \ref{Lem2}, respectively. In Lemma \ref{Lem3}, in particular, we have 
\begin{equation}\label{2sp1}
\int_{\R^n} \rho_\alpha (|\mbfu_\alpha|^2 - |\mbfc_\alpha|^2)\,d\mbfx = -|\mbfc_\alpha|^2\int_{\R^n} (\rho_\alpha - 1)\,d\mbfx, \quad (\alpha=i,e).
\end{equation}

\begin{proof}[Proof of Theorem \ref{MainThm2}]
 Suppose that $(\rho_i, \mbfu_i, \rho_e, \mbfu_e, \phi)$ is a classical solution of the system \eqref{2EP_S}--\eqref{kapa} satisfying $\rho_\alpha - 1, \mbfu_\alpha - \mbfc_\alpha, \phi \in L^1(\R^n)$ for $\alpha=i, e$. Following the proof of Theorem \ref{MainThm}, it is straightforward to see that 
\begin{equation}\label{2spe0}
\begin{split}
\sum_{\alpha=i,e}\frac{\eta_\alpha}{2}\int_{\R^n} \rho_\alpha (|\mbfu_\alpha|^2 - |\mbfc_\alpha|^2)\,d\mbfx  
& = \sum_{\alpha=i,e} \lim_{r \to \infty} \frac{\eta_\alpha}{2n} \int_{\R^n} (\mbfx \cdot \nabla \rho_\alpha) (|\mbfc_\alpha|^2- |\mbfu_\alpha|^2 ) \mu_r\,d\mbfx \\
& =\sum_{\alpha=i,e} \lim_{r \to \infty} \frac{1}{n} \int_{\R^n} (\mbfx \cdot \nabla \rho_\alpha) (\kappa_\alpha\phi + P_\alpha(\rho_\alpha)) \mu_r\,d\mbfx \\
& = - \int_{\R^n} (\rho_i - \rho_e)\phi\,d\mbfx + \frac{n-2}{2n}\int_{\R^n} |\nabla\phi|^2\,d\mbfx \\
& \quad - \int_{\R^n} \tilde{P}_i(\rho_i) + \tilde{P}_e(\rho_e) \,d\mbfx.
\end{split}
\end{equation}
Combined with \eqref{2sp3}, \eqref{2sp4}, and \eqref{2sp1}, it follows from \eqref{2spe0} that  
\begin{equation}\label{2spe1}
\frac{1}{2}(|\mbfc_i|^2 + \eta_e  |\mbfc_e|^2)\int_{\R^n} (\rho_i - 1) \,d\mbfx  =  \frac{n+2}{2n}\int_{\R^n} |\nabla\phi|^2\,d\mbfx +  \int_{\R^n} \tilde{P}_i(\rho_i) + \tilde{P}_e(\rho_e) \,d\mbfx.
\end{equation}

We recall the non-negativity of $\tilde{P}_\alpha$. If $\rho_i=\rho_e=1$, then any smooth $\phi$ satisfying \eqref{2EP_Sp} and \eqref{2EP_S5} must be identically $0$ by Liouville's theorem, and moreover, $\mbfu_\alpha=\mbfc_\alpha$ by Lemma \ref{Lem3}  (see \eqref{5}). If $\phi=0$, then $\rho_i = \rho_e =1$ or $\rho_i = \rho_e \neq 1$. Hence, we conclude from \eqref{2spe1} that for the cases $\mbfc_i=\mbfc_e=\mbf0$ and $|\mbfc_i|=\mu_e=0$, $(\rho_\alpha,\mbfu_\alpha,\phi)$ must be trivial, and for the case $|(\mbfc_i,\mbfc_e)| \neq 0$ and $\mu_e>0$, any nontrivial $(\rho_\alpha,\mbfu_\alpha,\phi)$  must satisfy 
\[
\int_{\R^n} (\rho_i - 1)\,d\mbfx > 0.
\] 

On the other hand, substituting  \eqref{2EP_Sp1} into \eqref{2spe0}, and using $\tilde{P}_\alpha \geq 0$ and \eqref{2sp4} (see also \eqref{6}), we obtain
\begin{equation*}
\begin{split}
0 
& = \int_{\R^n} \tilde{P}_i(\rho_i) - \rho_i P_i(\rho_i) + \tilde{P}_e(\rho_e)  - \rho_e P_e(\rho_e)\,d\mbfx - \frac{n-2}{2n}\int_{\R^n} |\nabla\phi|^2\,d\mbfx \\
& \leq \int_{\R^n} \tilde{P}_i(\rho_i) - \rho_i P_i(\rho_i) + \tilde{P}_e(\rho_e)  - \rho_eP_e(\rho_e)\,d\mbfx   \\
& = \int_{\R^n}  (1-\gamma_i)\tilde{P}_i(\rho_i) - K_i\gamma_i(\rho_i-1) + (1-\gamma_e)\tilde{P}_e(\rho_e) - K_e\gamma_e(\rho_e-1)  \,d\mbfx \\
& \leq -K_i\gamma_i\int_{\R^n}(\rho_i - 1)\,d\mbfx  - K_e\gamma_e \int_{\R^n}(\rho_e - 1)\,d\mbfx \\
& = - (K_i\gamma_i + K_e\gamma_e)\int_{\R^n}(\rho_i - 1)\,d\mbfx.
\end{split}
\end{equation*}
Since $(K_i, K_e) \neq (0,0)$ and $\gamma_i,\gamma_e \geq 1$, we can divide the above inequality by $-(K_i\gamma_i+K_e\gamma_e)<0$ to obtain a contradiction. We complete the proof.
\end{proof}

\section{Appendix}
\subsection{Derivation of the three-dimensional KP-I and KP-II equations}
We present a formal derivation of the KP equations from the following system
\begin{equation}\label{EP_A}
\left\{
\begin{array}{l l}
\partial_{t} \rho + \nabla \cdot (\rho \mathbf{u}) = 0, &  \\ 
\rho\left(\partial_{t} \mathbf{u}  + (\mathbf{u} \cdot \nabla ) \mathbf{u} \right) +   \nabla  p(\rho) = - \rho\nabla  \phi, &  ( \mathbf{x}, \; t\ge 0, \; n=2,3),\\
\mu\Delta \phi = \rho - e^\phi,  &\\
p(\rho)=K\rho^\gamma,  &  (K \geq 0, \; \gamma \geq 1).
\end{array} 
\right.
\end{equation}
where  $\rho>0$, $\mbfu \in \mathbb{R}^n$, $\phi \in \mathbb{R}$, and $\mu = \pm 1$. The system \eqref{EP} corresponds to the case $\mu=-1$ from which the KP-II equation is derived.

Introducing the scaling
\[
\bar{t} = \veps^{3/2}t, \quad \bar{x}_1 = \veps^{1/2}(x_1-Vt), \quad (\bar{x}_2,\bar{x}_3) = \veps(x_2,x_3), \quad \veps >0,
\]
\eqref{EP_A} becomes (neglecting the bar symbol)
\begin{equation}\label{EP_A1}
\left\{
\begin{array}{l l}
\veps\partial_{t} \rho - V\partial_{x_1}\rho + \partial_{x_1}(\rho u_1) + \veps^{1/2} \partial_{x_2}(\rho u_2) + \veps^{1/2} \partial_{x_3}(\rho u_3) = 0, \\ 
(\veps \partial_t - V\partial_{x_1})u_1 + (u_1\partial_{x_1} + \veps^{1/2}u_2\partial_{x_2} + \veps^{1/2}u_3\partial_{x_3})u_1 + \frac{\partial_{x_1}p(\rho)}{\rho} = -\partial_{x_1}\phi, \\
(\veps \partial_t - V\partial_{x_j})u_j + (u_1\partial_{x_1} + \veps^{1/2}u_2\partial_{x_2} + \veps^{1/2}u_3\partial_{x_3})u_j \\
\quad + \frac{\veps^{1/2}\partial_{x_j}p(\rho)}{\rho} = -\veps^{1/2}\partial_{x_j}\phi, \\
\mu \big(\veps \partial_{x_1}^2 + \veps^2 (\partial_{x_2}^2+\partial_{x_3}^2)\big)\phi = \rho - e^\phi,
\end{array} 
\right.
\end{equation}
where $j=2,3$. We assume that
\begin{equation}\label{FormalExp}
\rho = 1 + \sum_{j=1}^\infty \veps^j \rho^{(j)}, \quad u_1 =  \sum_{j=1}^\infty \veps^j u_1^{(j)}, \quad u_i =  \sum_{j=1}^\infty \veps^{j+1/2} u_i^{(j)},  \quad \phi = \sum_{j=1}^\infty \veps^j \phi^{(j)},
\end{equation}
and then we substitute the formal expansion \eqref{FormalExp} into \eqref{EP_A1}.

At the order of $\veps$, we have
\begin{equation}\label{EP_A2}
\left\{
\begin{array}{l l}
-V\partial_{x_1}\rho^{(1)} + \partial_{x_1}u_1^{(1)} = 0, \\ 
-V\partial_{x_1}u_1^{(1)} + K\gamma \partial_{x_1}\rho^{(1)} = -\partial_{x_1}\phi^{(1)} , \\
0 = \rho^{(1)} - \phi^{(1)}.
\end{array} 
\right.
\end{equation}
We let 
\begin{equation}\label{V}
V^2 = K\gamma + 1
\end{equation}
so that \eqref{EP_A2} has a nontrivial solution. Moreover, integrating \eqref{EP_A2} in $x_1$, we get
\begin{equation}\label{EP_A3}
u_1^{(1)} = V\rho^{(1)}, \quad \rho^{(1)} = \phi^{(1)}.
\end{equation}

At the order of $\veps^{3/2}$, we get 
\begin{equation}\label{EP_A4}
- V\partial_{x_1}u_j^{(1)} + K\gamma\partial_{x_j}\rho^{(1)} = -\partial_{x_j}\phi^{(1)}, \quad (j=2,3),
\end{equation}
and from \eqref{EP_A3} and \eqref{EP_A4}, we also obtain
\begin{equation}\label{EP_A5}
-\partial_{x_1}u_j^{(1)} + V\partial_{x_j}\rho^{(1)} = 0, \quad (j=2,3).
\end{equation}

At the order of $\veps^2$, we have 
\begin{subequations}\label{EP_A6}
\begin{align}[left = \empheqlbrace\,] 
& \partial_t\rho^{(1)} - V\partial_{x_1}\rho^{(2)} + \partial_{x_1}\left(u_1^{(2)} + \rho^{(1)}u_1^{(1)} \right) + \partial_{x_2}u_2^{(1)} +\partial_{x_3}u_3^{(1)} = 0, \label{EP_A61} \\
& \partial_t u_1 -V \partial_{x_1}u_1^{(2)} + u_1^{(1)}\partial_{x_1}u_1^{(1)} + K\gamma \partial_{x_1}\big( \rho^{(2)} + \frac{\gamma-2}{2}(\rho^{(1)})^2 \big) = -\partial_{x_1}\phi^{(2)} ,  \label{EP_A62} \\
&  \mu \partial_{x_1}^2\phi^{(1)} = \rho^{(2)} - \left( \phi^{(2)} + \frac{1}{2}(\phi^{(1)})^2 \right). \label{EP_A63}
\end{align}
\end{subequations}

We differentiate \eqref{EP_A63} in $x_1$ and then substitute it into \eqref{EP_A62}. Multiplying \eqref{EP_A61} by $V$, and then adding the resulting equations together, it is straightforward to see that the terms involving $(\rho^{(2)}, u_1^{(2)})$ are cancelled thanks to \eqref{V}. Then, using \eqref{EP_A3}, we have
\begin{equation*}
\partial_t \rho^{(1)} + \frac{K\gamma(\gamma+1) + 2}{2V}\rho^{(1)}\partial_{x_1}\rho^{(1)} - \frac{\mu}{2V}\partial_{x_1}^3\rho^{(1)} + \frac{1}{2}(\partial_{x_2}u_2^{(1)} + \partial_{x_3}u_3^{(1)}) = 0.
\end{equation*}
By differentiating the above equation in $x_1$, and then using \eqref{EP_A5}, we finally obtain
\[
\partial_{x_1}\left(\partial_t \rho^{(1)} + \frac{K\gamma(\gamma+1) + 2}{2V}\rho^{(1)}\partial_{x_1}\rho^{(1)} - \frac{\mu}{2V}\partial_{x_1}^3\rho^{(1)} \right) + \frac{V}{2}( \partial_{x_2}^2 + \partial_{x_3}^2)\rho^{(1)} = 0,
\]
which is the KP-II equation for $\mu=-1$ and the KP-I equation for $\mu=+1$  (see \eqref{KP}) after a suitable normalization.

\subsection{Leading order vector field is irrotational}
We derive that the curl of $\mbfu^{(1)}$ vanishes, where $\mbfu^{(1)}:=(u_1^{(1)},u_2^{(1)},u_3^{(1)})$ is the leading order term of the formal expansion \eqref{FormalExp}. 

From \eqref{EP_A3} and \eqref{EP_A5}, we see that 
\begin{equation}\label{Curl1}
\partial_{x_1}u_2^{(1)} = V\partial_{x_2}\rho^{(1)} = \partial_{x_2}u_1^{(1)}, \quad \partial_{x_1}u_3^{(1)} = V  \partial_{x_3} \rho^{(1)} = \partial_{x_3}u_1^{(1)}.
\end{equation}
By taking $\partial_{x_2}$ of \eqref{EP_A5} for $i=3$, we have
\[
\partial_{x_2}\partial_{x_1}u_3^{(1)} = V \partial_{x_2}\partial_{x_3}\rho^{(1)} =  \partial_{x_3}\partial_{x_2}u_1^{(1)} =  \partial_{x_3}\partial_{x_1}u_2^{(1)}.
\]
Integrating the above identity in $x_1$, we get
\begin{equation}\label{Curl2}
\partial_{x_2}u_3^{(1)} = \partial_{x_3}u_2^{(1)}.
\end{equation}
Combining \eqref{Curl1} and \eqref{Curl2}, we obtain  $\nabla \times \mbfu^{(1)} = \mbf0$.

\section*{Acknowledgement}
JB is supported by the National Research Foundation of Korea grant funded by the Ministry of Science and ICT (2022R1C1C2005658). DK is partly supported by JSPS KAKENHI Grant Numbers JP21K18586 and JP20K14344. The authors would like to express their gratitute to Jean-Claude Saut for many stimulating conversations on the topics related to the manuscript. JB wishes to thank the Laboratoire de Mathématiques at the Université Paris-Saclay, where the manuscript was completed, and Danielle Hilhorst for the invitation and hospitality.

 \end{document}